\newtheorem{Proof}{Proof.}
 
 \newenvironment{proof}{\begin{Proof}\rm}{\hfill $\Box$ \end{Proof}}
 
 \newtheorem{theo}{Preservation and Definability Theorem.}

 \newtheorem{theorem}{Theorem.}

 \documentstyle[leqno,11pt]{article}
\input amssymb.sty
\input amssymb.sty
\begin{document}
\title{A definitional view of Vogt's variant of the Mazur-Ulam theorem}
\author{Victor Pambuccian}
\date{corrected version of paper published in {\em Algebra, Geometry \& Their Applications. Seminar Proceedings} \bf{1} (2001), 5--10. }
\maketitle

\section{Introduction}

Perhaps the earliest result in the discipline that has come to be known
as {\em characterizations of geometric transformations by
weak hypotheses\/} is the theorem, proved by {\sc S. Mazur} and
{\sc S. Ulam} \cite{mu}, that surjective isometries between real normed spaces
are affine transformations. It has been widely generalized,
one of these generalizations being {\sc A. Vogt}'s \cite{vog} theorem
that equidistance-preserving transformations betweeen real normed spaces
of dimension $\geq 2$ are affine similarities.
{\sc F. Skof} \cite{sko} has provided conditions
for the conclusion of Vogt's theorem to hold without the
surjectivity requirement.

The aim of this note is to rephrase the theorems of
{\sc Vogt} and {\sc Skof} as theorems about the definability of a certain
geometric notion in terms of another one. To be precise,
whenever a theorem tells us that a map that preserves a certain
relation $\varrho$ must preserve another one $\varrho'$ as well,
we are told that $\varrho'$ is implicitly definable
in terms of $\varrho$. If the theorem can be phrased in a logical
language that satisfies  {\sc Beth}'s theorem, such as first-order logic
or ${\sf L}_{\omega_1\omega}$, then we know that there  must
be an explicit definition of $\varrho'$ in terms of
$\varrho$.  This fact is best expressed as (cf.\
\cite[Th.\ 6.6.4, Ex.\ 6.6.2]{hod}, \cite{mak}, \cite{kei})

\begin{theo}
Let ${\sf L}\subseteq {\sf L}^+$ be two first-order or
${\sf L}_{\omega_1\omega}$
languages containing a sign for an identically
false formula, ${\cal T}$ be a theory in ${\sf L}^+$,
and
$\varphi({\bf X})$ be an ${\sf L}^+$-formula in the free variables
${\bf X} = (X_1, \ldots, X_n)$. Then the following assertions are
$\varphi({\bf X})$ be an ${\sf L}^+$-formula in the free variables
${\bf X} = (X_1, \ldots, X_n)$. Then the following assertions are
equivalent:\\
$($i$)$ there is an  ${\sf L}$-formula $\psi({\bf X})$ $($which is positive existentia
l; positive existential, but negated equality is allowed; positive$)$
such
that ${\cal T}\vdash \varphi({\bf X}) \leftrightarrow \psi({\bf X})$;\\
$($ii$)$ for any $\mathfrak{A}, \mathfrak{B}$ $\in Mod({\cal T})$, and
each ${\sf L}$-isomorphism $({\sf L}$-homomorphism; ${\sf L}$-monomorphism;
 ${\sf L}$-epimorphism$)$
$f:\mathfrak{A}\rightarrow \mathfrak{B}$, the
following condition is satisfied:

if ${\bf c}\in \mathfrak{A}^n$ and $\mathfrak{A}\models \varphi({\bf c})$, then
$\mathfrak{B}$$\models \varphi(f({\bf c}))$.
\end{theo}

The theorems of {\sc Vogt} and {\sc Skof} are stated in the form (ii); the
aim of our paper is to state them purely syntactically, as in
(i), for we believe that the syntactic understanding  sheds
new light into the mechanics of these theorems.

\section{ The Axiomatic Set-Up}

To apply the above theorem, we need to find a theory ${\cal T}$,
among whose models are real normed linear spaces. Since we believe that
the Archimedean axiom is needed\footnote{ I know of no example of an
 equidistance preserving transformation between normed spaces with
 non-Archimedean ordered scalar fields that would not be an affine
 similarity, although the known proofs of {\sc Vogt}'s theorem
 depend heavily upon the Archimedeanity of the real field.}
, we cannot express ${\cal T}$ inside
first-order logic, but will have to work within ${\sf L}_{\omega_1\omega}$
(or within transitive closure logic, which is weaker than
${\sf L}_{\omega_1\omega}$, but not as easily readable\footnote{ One can easily
translate our definitions into TC logic, where the fact that there
is an explicit definition may be stronger, for I don't know whether
our Preservation and Definability Theorem remains valid in TC logic.}).

Our theorem will be weaker than {\sc Vogt}'s, since we require not
only the preservation of equidistance, but the
preservation of the negation of equidistance as well,
i.\ e.\ the map $f$ between two real normed vector spaces
needs to satisfy not only
$\|x-y\|=\|u-v\| \Rightarrow  \|f(x)-f(y)\|=\|f(u)-f(v)\|$,
but also the reverse implication, so we have $\Leftrightarrow$ instead
of $\Rightarrow$. It is an open problem whether {\sc Vogt}'s theorem remains
valid if we replace ${\mathbb R}$ with any Archimedean ordered
field. Both {\sc Vogt}'s proof \cite{vog} and the proof of {\sc Vogt}'s
 theorem in
\cite[p.\ 626-627]{os} use particular topological properties
of real normed spaces, such as the fact that the complement of
a sphere consists of two connected components, which are no longer
valid in our setting.

The theory ${\cal T}$ we are looking for is axiomatized in a language
${\sf L}_{\omega_1\omega}$, where ${\sf L}:= {\sf L}(B, \equiv)$, with
individual variables to be interpreted as {\em points}, and two
relation symbols, a ternary one $B$ and a quaternary one $\equiv$,
with with $B(abc)$ to be read as `point $b$ lies between $a$ and $c$
 (being allowed to be equal to one or to both endpoints)',
 and $ab\equiv cd$ to be read as
`$a$ is as distant from $b$ as $c$ is from $d$', or equivalently
`segment $ab$ is congruent to segment $cd$'.

Let $\Delta$ be an ${\sf L}(B)$ axiom system for
ordered Desarguesian affine spaces of dimension $\geq 2$.
One can easily obtain such an axiom system for spaces of
dimension $\geq 3$ by first rephrasing the
axioms given by {\sc Kusak} \cite{kus} vor affine spaces of dimension
$\geq 3$ in terms of collinearity $L$, with
$L(abc)$ to be read `points $a,b,c$ are collinear',
instead of parallelism, then replacing every occurrence of
$L(xyz)$ by $B(xyz)\vee B(yzx)\vee B(zxy)$, and adding order
axioms, e.\ g.\ as in \cite{pc}. Let $\theta$ be the sentence obtained
as the conjunction of all the  axioms mentioned above. Let $\theta'$ be the
conjunction of all the axioms for  Desarguesian ordered affine planes
from \cite{szm}. Then $\Delta$ may be chosen to be $\theta\vee \theta'$.

We further need \\
(a) axioms to ensure that $\equiv$ is a nondegenerate
equivalence relation between segments, i.\ e.\ $ab\equiv ba$,
$ab\equiv cd \wedge ab\equiv ef \rightarrow cd\equiv ef$,
$aa\equiv bb$, $ab\equiv cc \rightarrow a=b$;\\
(b) a segment transport axiom, such as
\[(\forall abc)(\exists d)(\forall e)\, [B(cad) \wedge
 ab\equiv ad \wedge (a\neq c \wedge B(cae)
  \wedge ab\equiv ae \rightarrow d=e)];\]
(c) an axiom stating that the affinely defined midpoint of any
segment $ab$ (as the intersection point of $ab$ with $cd$,
where $c$ and $d$ are two different points with
$ac\parallel bd$ and $ad\parallel bc$ ) is equidistant from $a$
and $b$;\\
(d) an axiom stating that if $abcd$ is a parallelogram with
$ab\parallel cd$ and $bc\parallel ad$, then $ab\equiv cd$ and
$bc\equiv ad$;\\
(e) an axiom stating that one obtains an isosceles triangle by
drawing a parallel to the base of an isosceles triangle, i.\ e.\
\[L(oab) \wedge L(oa'b')\wedge \neg L(oaa')\wedge b\neq o\wedge oa\equiv oa'
\wedge aa'\parallel bb'\rightarrow ob\equiv ob';\]
(f) the weak (equality is allowed)
triangle inequality, i.\ e.\  (the `triangle' is $abc$,
but notice that there is no condition of non-collinearity for
these points)
\[(B(ba'c)\vee B(bca'))\wedge ba\equiv ba'\wedge
B(ba'c') \wedge a'c'\equiv ac \rightarrow B(bcc');\]
(g) an axiom stating that there is a `triangle' whose sides are
three given segments  that satisfy the weak triangle inequality;\\
(h)  with the relation $\leq$  of inequality
between the lengths of segments defined, as in Schnabel \cite{schna}, by
\[ab\leq cd:\leftrightarrow (\forall m)(\exists s)\,
 [cm\equiv dm \rightarrow ab\equiv cs \wedge cm\equiv sm]\]
the axiom stating that any two segments are comparable under $\leq$,
i.\ e.\ $ab\leq cd \vee cd\leq ab$;\\
(i) the Archimedean axiom, which is the only one whose expression  requires infinitary logic,
and which may be written, with $P(abcd)$ standing for $a, b, c, d$ are
the four vertices of a parallelogram, i.\ e.\
 $ab\parallel cd$  and $bc\parallel ad$, as:
\begin{eqnarray*}
& & (\forall abx_1d)\, a\neq b \rightarrow
\bigvee_{n=2}^{\infty}\{(\exists x_2\ldots x_ny_1\ldots y_{n-1})\,
 x_1x_2\equiv ab \wedge (B(x_1x_2d)\vee B(x_1dx_2))\\
 & &  \bigwedge_{i=1}^{n-1}[P(x_ix_{i+1}y_{i+1}y_i)
 \wedge P(y_iy_{i+1}x_{i+2}x_{i+1})]\wedge B(x_1dx_n)\}.
 \end{eqnarray*}
Let $\Sigma$ denote the axiom system consisting of $\Delta$ as well as of the axioms
mentioned above.

 We can coordinatize models of $\Sigma$ as usual  and turn them
into right vector spaces with Archimedean
ordered fields as scalars (given that Archimedean ordered skew fields are commutative), with $B$ having its usual analytic
interpretation, and
where to any vector ${\bf x}$ we may associate a {\em norm} $\|{\bf x}\|$,
a non-negative element
of the scalar skew field, such that $\|{\bf x}\|=0$ iff ${\bf x} = {\bf
0}$,
 $\|{\bf x}\lambda\| =\|{\bf x}\|\cdot |\lambda|$, and
 $\|{\bf a} + {\bf b}\| \leq\ |{\bf a}\| + \ |{\bf b}\|$. For any vectors
${\bf a}$
and ${\bf b}$ we may thus define a {\em distance} by $d({\bf a}, {\bf
b}):=\|{\bf a - b}\|$,
and we have $ab\equiv cd$ iff $d({\bf a}, {\bf b})= d({\bf c}, {\bf d})$.

\section{ Definitional understandings of {\sc Vogt}'s and {\sc Skof}'s theorems}

We do not know whether the original variant of {\sc Vogt}'s theorem still
holds for models of $\Sigma$, i.\ e.\ if $\equiv$-preserving surjections
between models of $\Sigma$ also preserve the betweenness relation $B$.

We shall prove the following theorem, the first part of which
 is the definitional counterpart
of the weak variant of {\sc Vogt}'s theorem, the second part corresponding
to {\sc Skof}'s variant of {\sc Vogt}'s theorem, for
mappings that are not assumed surjective.
\begin{theorem}
$B$ and $\neq$ are definable in terms of $\equiv$,
and $B$ is definable by positive existential formulas in terms of $\equiv$
and $\neq$, the definitions being theorems of $\Sigma$.
\end{theorem}

\begin{proof}

We first show how to define the relation $\equiv_2$in terms of $\equiv$,
with
${\bf ab}\equiv_2 {\bf cd}$ to be interpreted as $d({\bf a}, {\bf b})= 2d({\bf c},
{\bf d})$,
by
\[ab\equiv_2 cd:\leftrightarrow (\exists e)\, ae\equiv cd \wedge be\equiv
cd
\wedge [(\forall xy)\, (xa\equiv xb \wedge ya\equiv yx) \rightarrow
(\exists z)\, zc\equiv xy \wedge zd\equiv xy].\]
We can now define the midpoint relation.  To this end, we recursively
 define a sequence of
relations $\varphi_n$  by\footnote{ For the rationale behind the $\varphi_n$'s
and their function in defining $B$ cf. \cite[\S 10.4]{ac}.}
\begin{eqnarray*}
\varphi_0(a,b,x)& := &  xa\equiv xb\wedge ab\equiv_2
xa,
\mbox{ and }\\
  \varphi_{n+1}(a,b,x):& \leftrightarrow & \varphi_n(a,b,x) \wedge
[(\forall x_3)(\exists x_1x_2y)\,
\varphi_0(a,b,x_3)\\
& & \rightarrow\bigwedge_{i=1}^2 \varphi_0(a,b,x_i)  \wedge xy\equiv_2 x_3x\wedge xy\leq x_1x_2]
\end{eqnarray*}
The definition of the midpoint relation $M$, with  $M({\bf abc})$
to be interpreted
as ${\bf a} + {\bf c} = 2{\bf b}$ with ${\bf a}\neq {\bf c}$, is
\[M(abc)\leftrightarrow a\neq c \bigwedge_{n=0}^\infty \varphi_n(a,c,b).\]
Let (the conjuncts being considered to exist only insofar as the indices
$i$ of the corresponding $x_i$ exist, i.\ e.\  if $n$ is large enough):
\begin{eqnarray*}
\alpha_n(a,b,x_{n-1}): &\leftrightarrow & a\neq b\wedge (\exists x_1\ldots x_{n-2})\,
M(abx_1)  \wedge M(bx_1x_2)\bigwedge_{i=2}^{n-2} M(x_{i-1}x_ix_{i+1}),\\
\beta_k(a,b,y_k): & \leftrightarrow & a\neq b \wedge (\exists y_1\ldots y_{k_1})\,
M(ay_1b)\bigwedge_{i=1}^{k-1}M(ay_{i+1}y_i).
\end{eqnarray*}
The two formulas are to be interpreted as: $\alpha_n({\bf a,b,x})$ iff
${\bf a}\neq {\bf b}$ and ${\bf x}$ is a point on ray
$\stackrel{\rightarrow}{{\bf ab}}$, such that $d({\bf a}, {\bf x})=
n\cdot d({\bf a}, {\bf b})$, and $\beta_k({\bf a,b,y})$ iff ${\bf a}\neq {\bf b}$
and ${\bf y}$ is a point on segment
${\bf ab}$, such that $d({\bf a}, {\bf y})
= 2^{-k}\cdot d({\bf a}, {\bf b})$.

Let now
\[\psi_{n,k}(a,b,c,d): \leftrightarrow a\neq b \wedge c\neq d\wedge
(\exists euv)\, ce\equiv au \wedge de\equiv av \wedge
\beta_k(a,b,v)\wedge \alpha_n(a, v, u).\]
The interpretation of $\psi_{n,k}({\bf a,b,c,d})$ is that
${\bf a}, {\bf b}$ as well as ${\bf c}, {\bf d}$ are different,
and that there is a point ${\bf e}$ such that $d({\bf c}, {\bf e})=
n2^{-k}\cdot d({\bf a}, {\bf b})$ and
$d({\bf d}, {\bf e}) = 2^{-k}\cdot d({\bf a}, {\bf b})$. By the triangle
inequality this implies that\footnote{ We are not pedantic about writing the scalar
multipliers to the right, since all of them are rational scalars, and thus
it is irelevant on which side they are written.}
\[\frac{(n-1)\cdot d({\bf a}, {\bf b})}{2^{k}}\leq d({\bf b}, {\bf c})
\leq \frac{(n+1)\cdot d({\bf a}, {\bf b})}{2^{k}}.\]
We are now ready to define
\[\gamma(a,b,c):\leftrightarrow  \bigwedge_{k=1}^\infty \{\bigvee_{n=1}^\infty
[\psi_{n,k}(a,b,b,c)\wedge \psi_{n+2^k,k}(a,b,a,c)]\},\]
to be interpreted as `${\bf a}, {\bf b}, {\bf c}$
are three different points, such that
$d({\bf a}, {\bf b}) + d({\bf b}, {\bf c}) = d({\bf a}, {\bf c})$'.

We are finally ready to define $B$. Its definition is:
\begin{eqnarray*}
B(abc) & \leftrightarrow & a=b \vee b=c \vee \{\bigwedge_{n=1}^\infty
[(\exists m_1\ldots m_{2^n-1}) \bigwedge_{i=2}^{2^n-2}
M(m_{i-1}m_im_{i+1})\wedge M(am_1m_2)\\
& & \wedge M(m_{2^n-2}m_{2^n-1}c) \wedge (\bigvee_{i=1}^{2^n-2}
(\gamma(m_i,b,m_{i+1})\vee \gamma(a,b, m_1)\vee \gamma(m_{2^n-1},b,c))]\}
\end{eqnarray*}

We may now define $\neq$, by first defining
\[\delta(z_0,x,z_n):\leftrightarrow (\exists z_1\ldots z_{n-1})\,
\bigwedge_{i=0}^{n-1}
z_iz_{i+1}\equiv z_0x,\]
with $\delta({\bf a,b,c})$ to be interpreted as $d({\bf a}, {\bf c})
\leq n\cdot d({\bf a}, {\bf b})$. We have
\[x\neq y:\leftrightarrow (\forall z)\, \bigvee_{n=2}^\infty \delta(x,y,z).\]
\end{proof}

\end{document}